\documentclass[12pt,reqno]{amsart}
\usepackage{amsmath, amssymb, amsfonts}
\numberwithin{equation}{section}
\newtheorem{thm}{Theorem}[section]
\newtheorem{prop}[thm]{Proposition}
\newtheorem{lem}[thm]{Lemma}
\newtheorem{dfn}[thm]{Definition}
\newtheorem{remark}[thm]{Remark}

\usepackage[a4paper, top = 1.2in, bottom = 1.2in, left = 1.2in, right = 1.2in]{geometry}
\numberwithin{equation}{section}

\newcommand{\A}{\mathbb{A}}
\newcommand{\C}{\mathbb{C}}
\newcommand{\N}{\mathbb{N}}
\newcommand{\Q}{\mathbb{Q}}
\newcommand{\R}{\mathbb{R}}

\newcommand{\Z}{\mathbb{Z}}

\newcommand{\f}{\mathbf{f}}
\newcommand{\g}{\mathbf{g}}

\newcommand{\mcO}{\mathcal{O}}

\newcommand{\mfa}{\mathfrak{a}}
\newcommand{\mfb}{\mathfrak{b}}
\newcommand{\mfc}{\mathfrak{c}}

\newcommand{\h}{\mathfrak{h}}

\newcommand{\mfm}{\mathfrak{m}}
\newcommand{\mfn}{\mathfrak{n}}

\newcommand{\Dif}{\mathfrak{D}}
\newcommand{\mfp}{\mathfrak{p}}
\newcommand{\mfq}{\mathfrak{q}}

\newcommand{\GL}{\mathrm{GL}}
\newcommand{\SO}{\mathrm{SO}}

\newcommand{\mrm}[1]{\mathrm{#1}}

\def\1{1\!\!1}

\newcommand{\pmat}[4]{ \begin{pmatrix} #1 & #2 \\ #3 & #4 \end{pmatrix}}

\def\dis{\displaystyle}

\title[On the nature of Fourier coefficients of Hilbert cusp forms]{On non-vanishing and sign changes of the Fourier coefficients of Hilbert cusp forms}

\author[T. Dalal]{Tarun Dalal}
\address[T. Dalal]{Department of Mathematics, Indian Institute of Technology Hyderabad, Kandi, Sangareddy 502285, INDIA.}
\email{ma17resch11005@iith.ac.in}

\author[N. Kumar]{Narasimha Kumar}
\address[N. Kumar]{Department of Mathematics, Indian Institute of Technology Hyderabad, Kandi, Sangareddy 502285, INDIA.}
\email{narasimha.kumar@iith.ac.in}

\date{\today}

\begin{document}
\begin{abstract}
In this article, we study (simultaneous) non-vanishing, (simultaneous) sign changes  of Fourier coefficients of (two) Hilbert  cusp forms, respectively. 
\end{abstract}

\maketitle

\section{Introduction}
The problem of non-vanishing and sign  changes of the Fourier coefficients of modular forms over a number field is an active area of research in number theory. 
For modular forms over $\Q$, there had been extensive  study of these problems by several mathematicians (cf. ~\cite{Mur83},~\cite{KM14},~\cite{GKR15}).
For modular forms over totally real number fields, a similar study has been initiated recently in~\cite{MT14},~\cite{KK}.



In~\S\ref{Sec0}, we shall recall the definition of Hilbert modular forms, their Fourier coefficients, and we will introduce some notations.

In~\S\ref{Sec1}, we shall study the simultaneous non-vanishing of Fourier coefficients of distinct primitive forms 
at powers of prime ideals (cf. Theorem~\ref{maintheorem3} in the text). 
In~\cite{GKP}, the authors proved that if $f$ and $g$ are two Hecke eigenforms of integral weights with $a_f(n), a_g(n) \in \R$, respectively, 
then for all but finitely many primes $p$, the set $\{m\in \N \mid a_f(p^m)a_g(p^m)\neq 0 \}$ has positive density. 
In~\cite[Theorem 3.1]{KK}, the authors extended this result to Hilbert primitive forms over $K$, by showing that the set in~\eqref{non-vanishingset} 
has positive density. In this article, we improve this result by showing that this density is at least $\frac{1}{2}$, when $[K:\Q]$ is odd.
In fact, we will show that the density can either be only $\frac{1}{2}$ or $1$.
The proof of this theorem is completely different from that of Theorem in~\textit{loc. cit.}.
Our proof depends on a generalization of the lemma~\cite[Lemma 2.2]{KRW07} or~\cite[Lemma 2.5]{MM07} 
to Hilbert modular forms (cf. Proposition~\ref{1:2} in the text). 

In~\S\ref{Sec2}, we shall study the sign change results for Fourier coefficients  of primitive forms over $K$ at powers of prime ideals, 
where $[K:\Q]$ is odd. In Proposition~\ref{Sign_Changes_Prime_Fixed}, for almost all prime ideals $\mfp$, we show that 
the Fourier coefficients at $\mfp^r(r \in \N)$ change signs infinitely often.
In Theorem~\ref{Sign_Changes_exponent_Fixed}, we show that a similar result hold by fixing an exponent 
and varying over prime ideals.


Let $a_f(n),a_g(n)\in \mathbb{R}$ be the Fourier coefficients of two non-zero cusp forms $f, g$, respectively, of same level but different weights. 
In~\cite[Theorem 1]{GKR15}, the authors showed that if $a_f(1)a_g(1) \neq 0$, then 
$a_f(n)a_g(n)(n \in \N)$ change signs infinitely many often.
In~\cite[Theorem 3.1]{KK}, the authors extend this result to Hilbert modular forms. 
In this article, we have improved the conditions of theorem in~\textit{loc. cit.}, so that it can be applied to a broader class of modular forms.





\section{Preliminaries}	 
\label{Sec0}
Let $F$ be a totally real number field and let $\mcO_F$ denote the integral closure of $\Z$ inside $F$.
In this section, we shall recall the basic definition of Hilbert modular forms over $F$ and it's Fourier coefficients for all 
integral ideals $\mfm \subseteq \mcO_F$ (for more details, see~\cite{Gar90},~\cite{Fre90}).

  Let $k=(k_1,\dots, k_n)\in \N^n$. For a non-archimedean place $\mfp$ of $F$, let $F_\mfp$ denote the completion of $F$ at $\mfp$.
  Let $\Dif_F$  denote the absolute different of $F$. Let $\mfa$ and $\mfb$ be integral ideals of $F$, and define a subgroup $K_\mfp(\mfa, \mfb)$ of $\GL_2(F_\mfp)$ as
  \[ K_\mfp(\mfa, \mfb)=\left\{\left(\begin{matrix} a & b \\ c & d \end{matrix}\right)\in \GL_2(F_\mfp)\, : \, \begin{matrix} a\in \mcO_\mfp, & b\in \mfa_\mfp^{-1}\Dif_\mfp^{-1}, & \\ c\in \mfb_\mfp\Dif_\mfp, & d\in\mcO_\mfp, & |ad-bc|_\mfp=1\end{matrix}\right\},\]
  where the subscript $\mfp$ means the $\mfp$-parts of given ideals. Furthermore, we put 
  \[ K_0(\mfa, \mfb)=\SO(2)^n\cdot\prod_{\mfp<\infty}K_\mfp(\mfa, \mfb) \quad \text{and} \quad W(\mfa, \mfb)=\GL_2^+(\R)^nK_0(\mfa, \mfb).\]
  In particular, if $\mfa=\mcO_F$, then we simply write $K_\mfp(\mfb):=K_\mfp(\mcO_F, \mfb)$, $W(\mfb):=W(\mcO_F, \mfb)$, etc.
  Then, we have the following disjoint decomposition of $\GL_2(\A_F)$:
  \begin{equation}\label{eqn:decomp}
  \GL_2(\A_F)=\cup_{\nu=1}^h\GL_2(F)x_\nu^{-\iota} W(\mfb),
  \end{equation}
  where $\dis x_\nu^{-\iota} =\left(\begin{matrix} t_\nu^{-1} & \\ & 1\end{matrix}\right)$ with $\{t_\nu\}_{\nu=1}^h$ taken to be a complete set of representatives of the narrow class group of $F$. We note that such $t_\nu$ can be chosen so that the infinity part $t_{\nu, \infty}$ is $1$ for all $\nu$. For each $\nu$, we also put
  \begin{align*}
  \Gamma_\nu(\mfb) &= \GL_2(F)\cap x_\nu W(\mfb)x_\nu^{-1} \\ 
  &= \left\{ \pmat{a}{t_\mu^{-1}b}{t_\nu c}{d}\in\GL_2(F): \, \begin{matrix} a\in \mcO_\mfp, & b\in \mfa_\mfp^{-1}\Dif_\mfp^{-1}, & \\ c\in \mfb_\mfp\Dif_\mfp, & d\in\mcO_\mfp, & |ad-bc|_\mfp=1\end{matrix}\right\}. 
  \end{align*} 
  
  Let $\psi$ be a Hecke character of $\A_F^\times$ whose conductor divides $\mfb$ and $\psi_\infty$ is of the form
  \[ \psi_\infty(x)={\rm sgn}(x_\infty)^k|x_\infty|^{i\mu},\]
  with $\mu\in\R^n$ and $\sum_{j=1}^n \mu_j=0$. We let $M_k(\Gamma_\nu(\mfb), \psi_\mfb, \mu)$ denote the space of all functions $f_\nu$ that are holomorphic on $\h^n$ and 
  at cusps, satisfying
  \[ f_\nu ||_k \gamma=\psi_\mfb(\gamma)\det \gamma^{i\mu/2}f_\nu \]
  for all $\gamma$ in $\Gamma_\nu(\mfb)$. We note that such a function $f_\nu$ has a Fourier expansion
  \[ f_\nu(z)=\sum_{\xi\in F}a_\nu(\xi) \exp(2\pi i \xi z)\]
  where $\xi$ runs over all the totally positive elements in $t_\nu^{-1}\mcO_F$ and $\xi=0$. 
  A Hilbert modular form is a cusp form, if for all $\gamma \in \GL^+_2(F)$, the constant term of $f||_k\gamma$
  in its Fourier expansion is $0$, and the space of cusp forms with respect to $\Gamma_{\nu}(\mfb)$ is denoted by $M_k(\Gamma_\nu(\mfb), \psi_\mfb, \mu)$.

  Now, put $\mathbf{f}:=(f_1,\dots,f_h)$ where $f_\nu$ belongs $M_k(\Gamma_\nu(\mfb), \psi_\mfb, \mu)$ for each $\nu$, and define $\mathbf{f}$ to be a function on $\GL_2(\A_F)$ as
  \[ \mathbf{f}(g)=\mathbf{f}(\gamma x_\nu^{-\iota}w):=\psi_\mfb(w^\iota)\det w_\infty^{i\mu/2}(f_\nu||_k w_\infty)(i\!\! i)\]
  where $\gamma x_\nu^{-\iota}w\in\GL_2(F)x_\nu^{-\iota}W(\mfb)$ as in (\ref{eqn:decomp}), and  $w^\iota:=\omega_0(^t w)\omega_0^{-1}$ with $\dis \omega_0=\left(\begin{matrix} & 1 \\ -1 & \end{matrix}\right)$. 
  The space of such $\mathbf{f}$ is denoted as $M_k(\psi_\mfb, \mu)=\prod_\nu M_k(\Gamma_\nu(\mfb), \psi_\mfb, \mu)$. Furthermore, the space consisting of all  $\mathbf{f}=(f_1,\dots,f_h)\in M_k(\psi_\mfb, \mu)$ satisfying
  \[ \mathbf{f}(sg)=\psi(s)\mathbf{f}(g) \quad \text{for any}\, s\in \A_F^\times \quad \text{and}\quad x\in\GL_2(\A_F)\]
  is denoted as $M_k(\mfb,\psi)$.  If $f_\nu \in S_k(\Gamma_\nu(\mfb), \psi_\mfb, \mu)$ for each $\nu$, then the space of such $\mathbf{f}$ is denoted by $S_k(\mfb,\psi)$.
  
  Let $\mfm$ be an integral ideal of $F$ and write $\mfm=\xi t_\nu^{-1}\mcO_F$ with a totally positive element $\xi$ in $F$. Then, we define the Fourier coefficients
  of $\f$ as
  \begin{equation}\label{eqn:coeff}
  C(\mfm,\mathbf{f}) :=\begin{cases} N(\mfm)^{\frac{k_0}{2}}a_\nu(\xi)\xi^{-(k+i\mu)/2} \quad & \text{if}\quad \mfm=\xi t_\nu^{-1}\mcO_F\subset \mcO_F \\
  0 & \text{if} \quad \mfm \, \text{is not integral} \end{cases} 
  \end{equation}
  where  $k_0=\max \{k_1,\dots,k_n\}$.  
  
  Throughout this article, by a primitive form $\f$ over $F$ of level $\mfb$, with character $\chi$ and weight $k$, 
  we mean $\f$ is a normalized Hilbert Hecke eigenform in $S_k^{\mrm{new}}(\mfb,\chi)$
  (cf. for the theory of new forms, please refer to~\cite{Shi78}).
 We let $F$ (resp., $K$) to denote a totally real number field (resp., of odd degree). 
Let $\mathbf{P}$ (resp., $\mathbb{P}$) denote the set of all prime ideals of $\mcO_F$ (resp., odd inertia degree). 
We shall use the same notations  $\mathbf{P}$ (resp., $\mathbb{P}$) for prime ideals (resp., odd inertia degree) of $\mcO_K$ as well and it shall be clear from the context.

Observe that, by ramification theory, for any prime $p \in \Z$, there exists a prime ideal $\mfp \subseteq \mcO_K$ over $p$ with odd inertia degree.
Furthermore, if $K$ is Galois, then every prime ideal of $\mcO_K$ has
odd inertia degree.



\subsection{Sato-Tate equi-distribution theorem}
In this section, we shall state the Sato-Tate equi-distribution theorem for non-CM primitive forms $\f$ (cf.~\cite[Theorem 3.3]{KKT18} which is a
re-formulation of~\cite[Corollary 7.17]{BGG11} for $\f$) 
in a way that shall be useful in our context.

Let $\f$ be a primitive form over $F$ of level $\mfc$, with trivial character and weight $2k$.
For any ideal $\mfa \subseteq \mcO_F$, define $\beta(\mfa,\f) := \frac{C(\mfa,\f)}{{N(\mfa)}^\frac{2k_0-1}{2}}$.
By Deligne's bound for $\f$, for any prime ideal $\mfp \nmid \mfc \Dif_F$, we have 
$\beta(\mfp,\f) \in [-2,2]$. Hence,
we can write 
\begin{equation} \label{theta}
  \beta(\mfp,\f) =  2 \cos \theta_\mfp(\f),
  \end{equation}
for some $\theta_\mfp(\f) \in [0, \pi]$.
Now, we shall recall the Sato-Tate equi-distribution theorem of Barnet-Lamb, Gee, and Geraghty (\cite[Corollary 7.17]{BGG11}).
  \begin{thm}
 \label{5.5}
 Let $\f$ be a non-CM primitive form over $F$ of level $\mfc$, with trivial character and weight $2k$.
 Then $\{\theta_\mfp(\f)\}_{\mfp \in \mathbf{P}, \mfp \nmid \mfc\Dif_F}$ is equi-distributed in $[0,\pi]$ with respect 
 to $\mu_{\mathrm{ST}}=\frac{2}{\pi}{\sin}^2\theta d\theta$. 
  In other words, for any sub-interval $I \subseteq [0,\pi]$, we have
\begin{equation}  
   \lim_{x\to \infty} \frac{\# \{ \mfp\in \mathbf{P} \mid \mfp\nmid\mfc\Dif_F, \mathrm{N}(\mfp)\leq x, \theta_\mfp(\f)\in I \}}{\# \{ \mfp\in \mathbf{P} \mid \mathrm{N}(\mfp) \leq x \}} 
   = \mu_{\mathrm{ST}}(I) = \frac{2}{\pi}\int_I {\sin}^2 \theta d\theta 
   \end{equation} i.e., the natural density of $S=\{\mfp \in \mathbf{P} \mid \mfp\nmid\mfc\Dif_F, \theta_\mfp(\f)\in I \}$  is $\mu_{\mathrm{ST}}(I)$.
 \end{thm}
%
\section{Non-vanishing of Fourier coefficients at prime powers}
\label{Sec1}
In this section, we shall prove a result concerning the simultaneous non-vanishing of Fourier coefficients of primitive forms at prime powers.
Before proving this result, we prove an important proposition, which a generalization of~\cite[Lemma 2.2]{KRW07} or ~\cite[Lemma 2.5]{MM07} to $K$.
\begin{prop}\label{1:2}
Let $\f$ be a primitive form over $K$ of level $\mfc$, with character $\chi$ and
weight $2k$.
   Then there exists an integer $M_\f \geq 1$ with $N(\mfc)\mid M_\f$ such that for any prime $p \nmid M_\f$ and 
   for any prime ideal $\mfp \in \mathbb{P}$ over $p$,
   we have either $C(\mfp,\f) = 0$ or $C(\mfp^r,\f) \neq 0$ for all $r \geq 1$. 
\end{prop}
\begin{proof}
Let $p$ be a prime number such that $p \nmid N(\mfc)$. Let $\mfp \in \mathbb{P}$ be a prime ideal of $\mcO_K$ over $p$ and $\mfp \nmid \mfc$. 
If $C(\mfp,\f)=0$, then there is nothing prove. If $C(\mfp, \f) \neq 0$, then we need to show that $C(\mfp^r,\f) \neq 0$ for all $r\geq 2$,
except for finitely many prime ideals $\mfp \in \mathbb{P}$.  

Suppose that $C(\mfp, \f) \neq 0$ but $C(\mfp^r,\f) = 0$ for some $r\geq 2$.
Since $\f$ is a primitive form, then by Hecke relations, we have
\begin{equation*}
    C(\mfp^{m+1}, \f) = C(\mfp,\f)C(\mfp^m,\f) - \chi(\mfp)N(\mfp)^{2k_0-1} C(\mfp^{m-1},\f).
\end{equation*}
These relations can be re-interpreted as
\begin{equation}
\sum_{r=0}^\infty C(\mfp^r,\f)X^r = \frac{1}{1-C(\mfp,\f)X+\chi(\mfp)
N(\mfp)^{2k_0-1}X^2}.
\end{equation}
Suppose that
\begin{equation*}
    1-C(\mfp,\f)X+\chi(\mfp)
N(\mfp)^{2k_0-1}X^2 = (1- \alpha(\mfp)X)(1-\beta(\mfp)X).
\end{equation*}
By comparing the coefficients, we get that
\begin{equation*}
    \alpha(\mfp) + \beta(\mfp) = C(\mfp, \f) \ \ \ \ \ \ \mathrm{and} \ \ \ \ \ \ \alpha(\mfp)\beta(\mfp) = \chi(\mfp)N(\mfp)^{2k_0-1} \neq 0,
    \end{equation*}
since $\mfp \nmid \mfc$ and hence $\chi(\mfp) \neq 0$. If $\alpha(\mfp) = \beta(\mfp)$, then 
\begin{equation*}
    C(\mfp^r,\f) = (r+1)\alpha(\mfp)^r\neq 0,
\end{equation*}
which cannot happen for any $r \geq 2$. So, $\alpha(\mfp)$ cannot be equal to $\beta(\mfp)$.
Then by induction, for any $r \geq 2$, we have the following
\begin{equation*}
    C(\mfp^r,\f) = \frac{\alpha(\mfp)^{r+1} - \beta(\mfp)^{r+1}}{\alpha(\mfp)- \beta(\mfp)}.
\end{equation*}
In this case, we have
\begin{equation*}
    C(\mfp^r,\f) = 0 \ \ \mathrm{if\ and\ only\ if} \ \ \Bigg(\frac{\alpha(\mfp)}{\beta(\mfp)}\Bigg)^{r+1} = 1,
\end{equation*}
which implies that the ratio $\frac{\alpha(\mfp)}{\beta(\mfp)}$ is a root of unity. Since $C(\mfp,\f) \neq 0$, 
we get that $\alpha(\mfp) = \zeta \beta(\mfp)$ where $\zeta$ is a root of unity and $\zeta \ne -1$ . By the product relation, we get that
$\alpha(\mfp)^2= \zeta\chi(\mfp){{N(\mfp)}^{2k_0-1}}$, hence $\alpha(\mfp) = \pm \gamma {N(\mfp)}^{{(2k_0-1)}/2} $, 
where $\gamma^2 = \zeta \chi(\mfp)$. Therefore,
\begin{equation*}
    C(\mfp,\f) = (1+\zeta^{-1})\alpha(\mfp) = \pm \gamma (1+\zeta^{-1}){N(\mfp)^{(2k_0-1)/2}} \ne 0.
\end{equation*}
In particular, $\mathbb{Q}(\gamma (1+\zeta^{-1})N(\mfp)^\frac{2k_0-1}{2}) \subseteq \mathbb{Q}(\f)$,
where $\Q(\f)$ is the field generated by $\{C(\mfm,\f)\}_{\mfm \subseteq \mcO_K}$ and by the values of the character $\chi$.
Since $\mfp \in \mathbb{P}$, $N(\mfp)= p^{f}$, where $f \in \N$ odd.
Hence, we have
\begin{equation}\label{5.1}
    \mathbb{Q}(\gamma(1+\zeta^{-1})p^{\frac{f(2k_0-1)}{2}}) \subseteq \mathbb{Q}(\f).
\end{equation}
Since $2k_0-1$, $f$ are odd, we have that 
\begin{equation}\label{5.2}
    \mathbb{Q}(\gamma(1+\zeta^{-1})\sqrt{p}) \subseteq \mathbb{Q}(\f).
\end{equation}
By~\cite[Proposition 2.8]{Shi78}, the field $\mathbb{Q}(\f)$ is a number field. Hence, the number of such primes $p$ are finite. Take $M_\f$ to be  the product of all such primes $p$ and $N(\mfc)$.
Thus, for any prime $p\nmid M_\f$ and for any prime ideal $\mfp \in \mathbb{P}$ over $p$, we have
either $C(\mfp,\f) = 0$ or $C(\mfp^r,\f)\neq 0$ for all $r\geq 1$.
\end{proof}
Observe that, the above proposition holds only for primes of $\mathbb{P}$.
This is because for primes of $\mathbf{P} \setminus \mathbb{P}$, in the above proof, ~\eqref{5.1} does not imply~\eqref{5.2}.
In this case, we may not be able to say that the number of such primes are finite. 
However, if $K$ is Galois over $\mathbb{Q}$, then the above proposition can be re-stated as:
\begin{lem}
Let $\f$ be as in Proposition~\ref{1:2}.
If $K$ is Galois over $\Q$, then
there exists an integer $M_\f \geq 1$ with $N(\mfc)\mid M_\f$ such that for any prime $p \nmid M_\f$ and 
   for any prime ideal $\mfp \in \mathbf{P}$ over $p$, 
   we have either $C(\mfp,\f) = 0$ or $C(\mfp^r,\f) \neq 0$ for all $r \geq 1$. 
\end{lem}

Now, we are in a position to state our main result of this section, which improves the result~\cite[Theorem 3.2]{KK}.
\begin{thm}
\label{maintheorem3}
Let $\f$ and $\g$ be two primitive forms over $K$ and 
of levels $\mfc_1,\mfc_2$, with characters $\chi_1 \ \mathrm{and} \ \chi_2$ and weights $2k,2l$,
 respectively.
 For any prime $p \nmid M_\f M_\g$, for any prime ideal $\mfp \in \mathbb{P}$ over $p$, the set
	\begin{equation}
	\label{non-vanishingset}
	A_{\mfp} := \{m\in \N| C(\mfp^m,\f)C(\mfp^m,\g) \neq 0\} 
	\end{equation}
        contains $2\N$, where  $M_\f$ and $M_\g$ are as in Lemma~\ref{1:2} for $\f,\g$, respectively. Moreover, the natural density of the set $A_{\mfp}$ 
        is either $\frac{1}{2}$ or $1$.
\end{thm}  
\begin{proof}
For any prime $p\nmid M_\f M_\g$, let $\mfp \in \mathbb{P}$ be a prime ideal over $p$.
If $C(\mfp, \f)C(\mfp, \g) \ne 0$, then by Lemma~\ref{1:2}, we have that 
$$\{m\in \N| C(\mfp^m,\f)C(\mfp^m,\g) \neq 0\} = \N.$$
In this case, the natural density of $A_{\mfp}$ is $1$.

Suppose at least one of $C(\mfp, \f)$ or $C(\mfp, \g)$ is zero, say $C(\mfp,\f)=0$.
By the Hecke relations for the primitive form $\f$
\begin{equation}\label{Hecke}
C(\mfp^m,\f) = -\chi_1(\mfp)N(\mfp)^{2k_0-1}C(\mfp^{m-2},\f),
\end{equation}
where $\chi_1(\mfp) \neq 0$, since $\mfp\nmid \mfc_1$. Hence, we see that the vanishing or non-vanishing of $C(\mfp^m, \f)$ depends only on $m \pmod 2$.
Therefore, $C(\mfp^{2m+1},\f) = 0$ (resp., $C(\mfp^{2m},\f) \ne 0$) as $C(\mfp,\f)=0$ (resp., $C(\mfp^2, \f) \neq 0$) for all $m \geq 1$.
Hence, we have that $$\{m\in \N \mid C(\mfp^m,\f) \neq 0\} = 2\N.$$ 
Arguing similarly for the primitive form $\g$,
we see that the set $\{m\in \N| C(\mfp^m,\g) \neq 0\}$ is either $\N$ or $2\N$ depends on whether $C(\mfp,\g) \ne 0$ or $C(\mfp,\g)= 0$, respectively.
So any of these cases, we get that $$\{m\in \N| C(\mfp^m,\f)C(\mfp^m,\g) \neq 0\} = 2\N.$$ 
In this case, the natural density of $A_{\mfp}$ is $\frac{1}{2}$. This proves the Theorem.
\end{proof}
In the view of above theorem, it is a natural question to ask is given a $\mfp \in \mathbb{P}$, how often the density of $A_{\mfp}$ is $1$?
The following proposition settles this question.

\begin{prop} 
\label{2.3.4}
Let $\f$ and $\g$ be same as in Theorem~\ref{maintheorem3}. If $\f,\g$ are non-CM eigenforms, 
then there exists a set $S \subseteq \mathbb{P}$  with natural density is $0$
such that 
\begin{equation}
\label{density}
A_{\mfp}=\{ m\in \mathbb{N} \mid C(\mfp^m, \f)C(\mfp^m,\g) \neq 0 \} = \mathbb{N} 
\end{equation}
for all prime ideals $\mfp \in \mathbb{P}$ outside of $S$.
\end{prop}
\begin{proof}
Define the set $S^{\prime} = \{ \mfp \in \mathbb{P} \mid C(\mfp, \f)C(\mfp, \g) = 0\}$. Clearly, we have 
$$\{ \mfp \in \mathbb{P} \mid C(\mfp, \f) = 0\} \subseteq S^{\prime} \subseteq \{ \mfp \in \mathbb{P} \mid C(\mfp, \f) =0 \} \cup \{ \mfp
 \in \mathbb{P} \mid C(\mfp, \g) =0 \}.$$ 
 By Theorem~\ref{5.5}, the natural density of $\{ \mfp \in \mathbf{P} \mid C(\mfp, \f) = 0\}$ is $0$ and hence 
 the natural density of $\{ \mfp \in \mathbb{P} \mid C(\mfp,\f) =0 \}$ is $0$. Similarly, for the eigenform $\g$ as well.
  Hence, the natural density of $S^{\prime}$ is $0$. 
 Therefore, the natural density of the set $S = S^{\prime} \cup \{ \mfp \in \mathbb{P} \mid \mfp \mid p \ \mathrm{and} \ p\mid M_\f M_\g \}$ is $0$. 
 For any $\mfp \not \in S$, by Lemma~\ref{1:2},  we have $C(\mfp^m, \f) C(\mfp^m, \g) \neq 0$ for all $m \geq 1$.
\end{proof}
We remark that, in the above result, if we assume $K$ is Galois, then~\eqref{density} holds for density $1$ set of primes in $\mathbf{P}$ (because, in this case $\mathbb{P}=\mathbf{P}$).


\section{Sign changes of Hilbert modular forms}
\label{Sec2}
In this section, we shall study the sign change results for the Fourier coefficients of primitive forms, and later we study the simultaneous sign changes 
for the Fourier coefficients of two non-zero Hilbert modular forms of different integral weights. 

\subsection{Sign changes}
In ~\cite[Theorem 1.1]{MT14}, the authors show that a non-zero Hilbert cusp form with real Fourier coefficients change signs infinitely often. 
In the next proposition, for primitive forms, we show that for almost all the primes $\mfp \in \mathbb{P}$,
the Fourier coefficients $\{C(\mfp^r,\f)\}_{r\in \mathbb{N}}$ change signs infinitely often.
\begin{prop}
\label{Sign_Changes_Prime_Fixed}
Let $\f$ be a primitive form over $K$ of level $\mfc$, trivial character and
weight $2k$.
Then, for all but finitely many $\mfp \in \mathbb{P}$, the Fourier coefficients $\{C(\mfp^r,\f)\}_{r \in \N}$ 
change signs infinitely often.
\end{prop}
\begin{proof}
Let $\mfp \in \mathbb{P}$ be a prime ideal  such that $C(\mfp^r,\f) \geq 0$ for all $r \gg 0$ (a similar argument holds in the other case as well).
Since $\f$ is primitive, by Hecke relations, we have
\begin{equation*}
    C(\mfp^{m+1}, \f) = C(\mfp,\f)C(\mfp^m,\f) - N(\mfp)^{2k_0-1} C(\mfp^{m-1},\f).
\end{equation*}
These Hecke relations can be re-interpreted as
\begin{equation}
\sum_{r=0}^\infty C(\mfp^r,\f)X^r = \frac{1}{1-C(\mfp,\f)X + N(\mfp)^{2k_0-1}X^2}.
\end{equation}
Suppose that
\begin{equation*}
    1-C(\mfp,\f)X+
N(\mfp)^{2k_0-1}X^2 = (1- \alpha(\mfp)X)(1-\beta(\mfp)X).
\end{equation*}
Then
\begin{equation}\label{222}
\sum_{r=0}^{\infty} C(\mfp^r,\f)X^r = (1- \alpha(\mfp)X)^{-1}(1-\beta(\mfp)X)^{-1},
\end{equation}
Comparing the coefficients we have $$\alpha(\mfp) + \beta(\mfp) = C(\mfp, \f) \ \  \mathrm{and} \ \  \alpha(\mfp)\beta(\mfp) = N(\mfp)^{2k_0-1},$$
where
\begin{equation}\label{224}
\alpha(\mfp), \beta(\mfp) = \frac{C(\mfp,\f) \pm \sqrt{C(\mfp,\f)^2 - 4N(\mfp)^{2k_0-1}}}{2}.
\end{equation}
For $s\in \mathbb{C}$, replacing $X$ by $N(\mfp)^{-s}$ in ~\eqref{222}, we get that
\begin{equation}\label{223}
\sum_{r=0}^{\infty}C(\mfp^r,\f)N(\mfp)^{-sr} = (1- \alpha(\mfp)N(\mfp)^{-s})^{-1}(1-\beta(\mfp)N(\mfp)^{-s})^{-1}.
\end{equation}
The above Dirichlet series converges for $\mrm{Re}(s) \gg 0$ and the coefficients are non-negative except for finitely many terms.
By Landau's theorem for Dirichlet series with non-negative terms, we get the series~\eqref{223} is either converges everywhere 
or it has a singularity at the real point of its abscissa of convergence.
The series has a pole at $s\in \mathbb{C}$ for which $N(\mfp)^s = \alpha(\mfp)$ or $N(\mfp)^s = \beta(\mfp)$ holds, hence the first case is not possible. 
Then the only possibility is that the series  has a singularity at the real point of its abscissa of convergence. In particular,
one of (and hence both of) $\alpha(\mfp)$ or $\beta (\mfp)$ must be real. Hence, we get that ${C(\mfp,\f)}^2 \geq 4N(\mfp)^{2k_0-1}$. However, by Deligne's bound for $\f$, we have 
\begin{equation}\label{Dbound}
 {C(\mfp,\f)}^2 \leq 4N(\mfp)^{2k_0-1}.
 \end{equation}
Therefore, 
  \begin{equation}
  \label{225} 
  C(\mfp,\f) = \pm 2N(\mfp)^{\frac{2k_0-1}{2}} \in \mathbb{Q}(\f).
\end{equation}  
Since $\mfp \in \mathbb{P}$, by~\eqref{225}, we get $\sqrt{p} \in \mathbb{Q}(\f)$, which can only happen for finitely many primes $p$.
This proves the lemma.
\end{proof}

In~\cite{KM14}, Kohnen and Martin remarked that the sign change results for the Fourier coefficients can also be proved by using sign changes
of $\sin(\theta)$. We elaborate this remark and reprove the above result. For this, we need to recall the following lemma (cf.  by~\cite[Proposition 5.1]{KK}
for a proof).
\begin{lem}
Let $\f$ be a primitive form over $F$ of level $\mfc$, with trivial character and weight $2k$.
  For any prime ideal $\mfp \nmid \mfc \Dif_F$, let $\theta_\mfp(\f) \in [0,\pi]$ be defined as in~\eqref{theta}.
 Then, for any $m \geq 1$, we have 
 \begin{equation}\label{normalized}
 \beta(\mfp^m,\f) = 
 \begin{cases}
 (-1)^m(m+1) & \mathrm{if} \ \theta_\mfp(\f) = \pi,\\
 m+1 & \mathrm{if} \ \theta_\mfp(\f) = 0,\\
 \frac{\sin((m+1)\theta_\mfp(\f))}{\sin\theta_\mfp(\f)} & \mathrm{if} \ 0<\theta_\mfp(\f) < \pi.
 \end{cases}
  \end{equation}
 \end{lem}
Now, we shall give another proof of Proposition~\ref{Sign_Changes_Prime_Fixed}.
\begin{proof}
For any $\mfp \in \mathbb{P}$, if $\theta_\mfp(\f) = 0$ or $\pi$, then 
$C(\mfp,\f) = \pm 2N(\mfp)^{\frac{2k_0-1}{2}} \in \mathbb{Q}(\f)$,
which can happen only for finitely many $\mfp  \in \mathbb{P}$. 
So, without loss of generality, we can assume that $0<\theta_\mfp(\f) < \pi$, hence $\sin(\theta_\mfp(\f))>0$.
By~\eqref{normalized}, we have
 $$C(\mfp^m,\f)\gtrless0 \iff \sin2\pi (m+1)\frac{\theta_\mfp(\f)}{2\pi}\gtrless 0.$$
Let $x=\frac{\theta_\mfp(\f)}{2\pi}$.
For any $j \in \N$, the lengths of the intervals $(\frac{2j}{2x}, \frac{(2j +1)}{2x})$ and $(\frac{(2j-1)}{2x}, \frac{2j}{2x} )$ are bigger than $1$, 
as $\frac{1}{2x} > 1$. Hence, there exists $n_j, m_j \in \Z$ such that $n_j+1 \in (\frac{2j}{2x}, \frac{2j +1}{2x})$ and $m_j +1 \in ( \frac{2j-1}{2x}, \frac{2j}{2x})$.
Therefore, we have $\sin ((n_j+1)\theta_\mfp(\f)) >0$ and $\sin((m_j+1)\theta_\mfp(\f)) <0$. This completes the proof.
 \end{proof}
In the above proposition, for a prime $\mfp \in \mathbb{P}$, we have studied the sign changes for $\{C(\mfp^r,\f)\}_{r \in \N}$.
Now, for a fixed $r \in \N$, we are interested in studying the sign changes for $\{C(\mfp^r,\f)\}_{\mfp \in \mathbf{P}}$. 

For primitive forms over $\mathbb{Q}$, this question has been studied in~\cite[Theorem 1.1]{MKV18}. 
In fact, they have computed the natural densities of these sets depending on $r$ is even or odd.
In this next theorem, 
we shall show that a similar result  holds for primitive forms over $F$, essentially by following the same approach. So, we shall
state the theorem and sketch a proof of it. To state it, we shall need the notion of natural density for a subset of prime ideals.
\begin{dfn}
 Let $F$ be a number field and $S \subseteq \mathbf{P}$ be a subset of prime ideals of $\mcO_F$. The natural density of $S$ defined as
 $$d(S)= \lim_{x\to \infty} \frac{\# \{\mfp \in S \mid \mathrm{N}(\mfp) \leq x\}}{\# \{\mfp \in \mathbf{P} \mid \mathrm{N}(\mfp)\leq x\}},$$ 
 if the limit exists.
 \end{dfn}

\begin{thm}
\label{Meher}
\label{Sign_Changes_exponent_Fixed}
Let $\f$ be a non-CM primitive form over $F$ of level $\mfc$, with trivial character and weight $2k$.
For any $m \geq 1$, we define  
$$\mathbf{P}(m)_{\gtrless 0} = \{\mfp \in \mathbf{P} \mid \mfp \nmid \mfc \Dif_F, C(\mfp^m,\f) \gtrless 0\}.$$
\begin{enumerate}
 \item If $m \equiv 1 \pmod 2$, then $$d(\mathbf{P}(m)_{>0})= d(\mathbf{P}(m)_{<0}) = \frac{1}{2}.$$
 \item If $m \equiv 0 \pmod 2$, then 
          $$d(\mathbf{P}(m)_{> 0}) = \frac{m+2}{2(m+1)} -\frac{1}{2\pi}\tan \Bigg(\frac{\pi}{m+1}\Bigg), \ \mrm{and}$$ 
           $$d(\mathbf{P}(m)_{<0}) = \frac{m}{2(m+1)} + \frac{1}{2\pi}\tan \Bigg(\frac{\pi}{m+1}\Bigg).$$
\end{enumerate}
In particular, then for any $m\in \mathbb{N}$, the sequence $\{C(\mfp^m,\f)\}_{\mfp\in \mathbf{P}}$ changes sign infinitely often. 
\end{thm}

\begin{proof}
By Theorem~\ref{5.5}, the natural density of 
$T= \{\mfp \in \mathbf{P} \mid \theta_\mfp(\f) = 0, \pi \} \cup \{\mfp \in \mathbf{P} \mid \mfp \mid \mfc \Dif_F\}$ 
is zero. By~\eqref{normalized}, we have the following equality
$$\mathbf{P}(m)_{\gtrless 0} = \{ \mfp \in \mathbf{P} \mid \mfp \not \in T, \sin((m+1)\theta_\mfp(\f)) \gtrless 0\}.$$  
If $m\equiv 0 \pmod 2$, then
$$ \sin ((m+1)\theta_\mfp(\f)) >0 \Leftrightarrow \theta_\mfp(\f) \in  S:= \bigcup_{j= 0}^\frac{m}{2} \Bigg(\frac{2j\pi}{m+1}, \frac{(2j+1)\pi}{m+1} \Bigg),$$ 
and
$$ \sin ((m+1)\theta_\mfp(\f)) <0 \Leftrightarrow \theta_\mfp(\f) \in  \bigcup_{j= 1}^\frac{m}{2} \Bigg(\frac{(2j-1)\pi}{m+1}, \frac{2j\pi}{m+1} \Bigg).$$ 
By Theorem~\ref{5.5}, the density  of $\mathbf{P}(m)_{>0}$ exists and $d(\mathbf{P}(m)_{>0}) = \mu_{\mathrm{ST}}(S)$, where
$\mu_{\mathrm{ST}}(S) = \frac{2}{\pi}\int_S {\sin}^2tdt$.
 The explicit calculation of
$\mu_{\mrm{ST}}(S)$ is exactly the same as that of~\cite[Theorem 1.1]{MKV18}. 
Again by Theorem~\ref{5.5}, we see that the natural density of $\{\mfp \in \mathbf{P} \mid \mfp\nmid \mfc \Dif_F, C(\mfp^m,\f)=0\}$ 
is $0$, hence we have $$d(\mathbf{P}(m)_{<0}) = 1- d(\mathbf{P}(m)_{>0}).$$
In the case of $m \equiv 1 \pmod 2$, a similar calculation in \textit{loc.cit.} works as well.
\end{proof}

\subsection{Simultaneous sign changes}
 In~\cite[Theorem 3.1]{KK}, the authors proved that, if $C(\mcO_F,\f)C(\mcO_F,\g) \neq 0$, then there exists infinitely many integral ideals such that the product of the Fourier coefficients 
 of $\f$ and $\g$ is positive (resp., negative). Now, we shall state the main theorem this section.
\begin{thm}
    \label{maintheorem}
   Let $\f$ and $\g$ be non-zero Hilbert cusp forms over $F$ of level $\mfc$, trivial character and different integral weights $k$, $l$,
   respectively. 
   Assume that for every ideal $\mfn \subseteq \mcO_F$, there exists an ideal $\mathfrak{r}  \subseteq \mcO_F$ such that $(\mfn,\mathfrak{r})=1$ such that $C(\mathfrak{r},\f)C(\mathfrak{r},\g) \neq 0$.
      Then there exist infinitely many ideals $\mfm \subseteq \mcO_F$ such that $C(\mfm,\f) C(\mfm,\g)>0$ and infinitely many ideals $\mfm \subseteq \mcO_F$ such that $C(\mfm,\f)C(\mfm,\g)<0$.
  \end{thm} 
\begin{remark}
In the above theorem, the condition of simultaneous non-vanishing of Fourier coefficients is required only to ensure that 
the $L$-function in~\eqref{positive:cofficients} is non-zero, otherwise there is no other reason for this assumption.
\end{remark}

The main idea in the proof of Theorem~\ref{maintheorem} comes from~\cite[Theorem 1.5]{KM18}, which mainly uses the following theorem of Pribitkin~\cite{Pri08}.
  \begin{thm}
   \label{lem1}
     Let $F(s) = \sum_{n=1}^{\infty} a_n e^{-s\lambda_n}$ be a non-trivial general Dirichlet series which converges somewhere, where the  sequence
     $\{a_n\}_{n=1}^{\infty}$ is  complex and  the  exponent sequence $\{\lambda_n\}_{n=1}^{\infty}$ is real and strictly increasing to $\infty$.
     If the function $F$ is holomorphic on the whole real line and has infinitely many real zeros, then there exist infinitely many
     $n \in \N$ such that $a_n > 0$ (resp., $a_n < 0$).
  \end{thm}

The following proposition is a melange of~\cite[Proposition 2.3]{Shi78} and~\cite[Page 124]{Pan91}. 
\begin{prop}
For any integral ideal $\mfq\subseteq \mcO_F$ and every $\f\in S_k(\mfc, \psi)$, there exists a unique element $\f|\mfq$ of $S_k(\mfq\mfc, \psi)$
such that $$C(\mfm,\f|\mfq)=C(\mfq^{-1}\mfm,\f),$$ and there exists an unique element $\f|U(\mfq)$ of $S_k(\mfq\mfc,\psi)$
such that $$C(\mfm,\f|U(\mfq))=C(\mfq\mfm,\f).$$
\end{prop} 
Before we proceed to prove Theorem~\ref{maintheorem}, we need the following proposition to construct new Hilbert modular forms out of the existing modular
form with some prescribed vanishing of Fourier coefficients at certain ideals (cf.~\cite[Proposition 4.5]{KK} for a proof).
\begin{prop}
\label{Key-Proposition}
Let $\f\in S_k(\mfc,\psi)$ and $\mfq$  be an integral ideal of $\mcO_F$.
Then $\g=\f-(\f|U(\mfq))|\mfq$ is a Hilbert cusp form of weight $k$ and level $\mfq^2\mfc$.
Further, it has the property that $C(\mfm\mfq,\g)=0$ and $C(\mfm,\g)=C(\mfm,\f)$, if $(\mfm,\mfq)=1$. 
\end{prop}

%
Now, we are ready to prove Theorem~\ref{maintheorem}.
 \begin{proof}
         First, we shall show that there exist infinitely many $\mfm\subseteq\mcO_F$ such that 
 	\begin{equation}
 	   \label{negative}
 	C(\mfm,\f)C(\mfm,\g)<0.
 	\end{equation}
 	A similar proof works for the other case as well, by replacing $\f$ by $-\f$. If~\eqref{negative} is not true, then there exist an ideal $\mfm^\prime\subseteq\mcO_F$ such that 
 	\begin{equation}\label{positive}
 	C(\mfm,\f)C(\mfm,\g)\geq 0
 	\end{equation}
 	for all $\mfm\subseteq\mcO_F$ with $N(\mfm)\geq N(\mfm^\prime)$. 
 	Set $\mfn:=\prod_{N(\mfp)\leq N(\mfm^\prime)} \mfp$, where $\mfp$ are prime ideals of $\mcO_F$. 
 	
 	Suppose $\f_1$ and $\g_1$ are Hilbert modular cusp forms obtained from $\f$ and $\g$ respectively,
 	by applying the Proposition~\ref{Key-Proposition} to $\f$ and $\g$ with the ideal $\mfn$. 
 	Clearly, $\f_1$ and $\g_1$ are also Hilbert cusp forms of  level $k$ and $l$ respectively, and of level  $\mfc_1= \mfq^2\mfc$.
 	 	For $s\in\C$ with $\mathrm{Re}(s)\gg 1$, the Rankin-Selberg $L$-function of $\f_1$ and $\g_1$ is defined by 
 	\begin{equation}\label{positive:cofficients}
 	R_{\f_1,\g_1}(s):=\sum_{\mfm\subseteq\mcO_F,(\mfm,\mfn)=1} \frac{C(\mfm,\f)C(\mfm,\g)}{N(\mfm)^s}.
 	\end{equation}
 	In above summation $C(\mfm,\f)C(\mfm,\g)\geq 0$, since, if $N(\mfm)\leq N(\mfm^\prime)$ then $\mfm=\prod_{\mfp_i|\mfn}\mfp_i^{e_i}$ implies  $(\mfm,\mfn)\neq1$. 
        The Rankin-Selberg $L$-function $R_{f_1,g_1}(s)$ is a non-zero function since there exists $\mfm$ with $(\mfm,\mfn)=1$ such that $C(\mfm,\f)C(\mfm,\g) \neq 0$, by hypothesis.

 	For $\mathrm{Re}(s)\gg 1$, we set 
 	$$L_{\f_1,\g_1}(s):=\zeta_F^{\mfc_1}(2s-(k_0+l_0)+2)R_{\f_1,\g_1}(s),$$
 	where $\zeta_F^{\mfc_1}(s)=\prod_{\mfp|\mfc_1, \mfp:\text{prime}}(1-N(\mfp)^{-s})\zeta_F(s)$, 
 	where $\zeta_F(s)=\sum_{\mfm\subseteq\mcO_F}N(\mfm)^{-s}$ is Dedekind zeta function of $F$. 
 	By the Euler expansion of Dedekind zeta function of $F$,  we get that
   \begin{align*}
   	\zeta_F^{\mfc_1}(s) =&\prod_{\mfp|\mfc_1, \mfp:\text{prime}}(1-N(\mfp)^{-s})\prod_{\mfp:\text{prime}}(1-N(\mfp)^{-s})^{-1}\\
   	                    &=\sum_{\mfm\subseteq\mcO_F,(\mfm,\mfc_1)=1}\frac{1}{N(\mfm)^{s}}=\sum_{n=1}^{\infty}\frac{a_n(\mfc_1)}{n^s},
   \end{align*}
   where $a_n(\mfc_1)$ is the number of integral ideals of norm $n$ that are co-prime to $\mfc_1$.
 	Hence, we can write
 	$$L_{\f_1,\g_1}(s)=\sum_{n=1}^{\infty}\frac{a_n(\mfc_1)n^{k_0+l_0-2}}{n^{2s}}\sum_{\mfm\subseteq\mcO_F,(\mfm,\mfn)=1} \frac{C(\mfm,\f)C(\mfm,\g)}{N(\mfm)^s}.$$
 	Now, we can re-write $$L_{\f_1,\g_1}(s)=\sum_{m=1}^\infty\frac{\mfb_m^{\mfc_1}(\f_1,\g_1)}{m^s}= \sum_{m=1}^{\infty}{\mfb_m^{\mfc_1}(\f_1,\g_1) e^{-s \log m}},$$
 	where $$\mfb_m^{\mfc_1}(\f_1,\g_1)=\sum_{n^2|m}\left(a_n(\mfc_1)n^{k_0+l_0-2}\sum_{(\mfm,\mfn)=1,N(\mfm)=m/n^2}C(\mfm,\f)C(\mfm,\g)\right).$$
        Define, for any $j$, $k_j^{\prime} := k_0 - k_j$, and similarly, define $l_j^{\prime}$. Now, look at the complete $L$-function, defined by the product
 	$$\Lambda_{\f_1,\g_1}(s)= \prod_{j=1}^n \Gamma\left(s+1+ \frac{k_j-l_j-k_0-l_0}{2}\right)\Gamma\left(s-\frac{k^{\prime}_j+l^{\prime}_j}{2}\right)L_{\f_1,\g_1}(s)$$ 
 	can be continued to a holomorphic function on the whole plane, since the weights are different (cf.~\cite[Proposition 4.13]{Shi78}).
 	As the $\Gamma$-function is extended by analytic continuation to all complex numbers except the non-positive integers, where the function has simple poles,
 	we get that that function $L_{\f_1,\g_1}(s)$ is also entire and has infinitely many real zeros because the $\Gamma$-factors have poles at non-positive integers. 
 	
        By Landau's Theorem for Dirichlet series with non-negative coefficients, it follows that the Dirichlet series $L_{\f_1,\g_1}(s)$ converges everywhere.
 	By Theorem~\ref{lem1}, there exist infinitely many $m \in \N$ such that $\mfb_m^{\mfc_1}(\f_1,\g_1) > 0 $ and there exist infinitely many
        $m \in \N$ such that $\mfb_m^{\mfc_1}(\f_1,\g_1) < 0$. This is a contradiction to the fact  $\mfb_m^{\mfc_1}(\f_1,\g_1)\geq 0$ for all $m$ 
        (this is because, by~\eqref{positive}, $C(\mfm,\f)C(\mfm,\g)\geq 0$ for all $(\mfm,\mfn)=1$). This completes the proof of Theorem~\ref{maintheorem}.
\end{proof}

In the following proposition, we compute the natural density of $n\in \N$
such that the product $C(\mfp^n,\f)C(\mfp^n,\g)$ have the same sign (resp., opposite sign).
For primitive forms over $\Q$, this a result due to Amri (cf. ~\cite[Theorem 1.1]{Amr18}).
\begin{prop}
\label{Amri}
Let $\f,\g$ be two distinct non-CM primitive forms over $F$ of levels $\mfc_1,\mfc_2$, with trivial characters, 
and weights $2k,2l$, 
respectively.
For any prime ideal $\mfp \in \mathbf{P}$ with $\mfp \nmid \mfc_1 \mfc_2 \Dif_F$, let $\theta_\mfp(\f), \theta_\mfp(\g) \in [0,\pi]$
be defined as in~\eqref{theta}.
Then, for a natural density $1$ set of primes $\mfp \in \mathbf{P}$, the linear independence of
$1, \frac{\theta_\mfp(\f)}{2\pi}, \frac{\theta_\mfp(\g)}{2\pi}$ over $\mathbb{Q}$ implies 
$$ \lim_{x \to \infty} \frac{\# \{n\leq x : C(\mfp^n,\f)C(\mfp^n,\g)\gtrless 0\} }{x} = \frac{1}{2}. $$
\end{prop}
\begin{proof}
By Theorem~\ref{5.5}, the natural density $\mfp \in \mathbf{P}$ such that $\theta_\mfp(\f),\theta_\mfp(\g) \in \{0,\pi\}$ is zero. 
Let $\mfp \in \mathbf{P}$ be a prime ideal such that $\theta_\mfp(\f),\theta_\mfp(\g) \in (0,\pi)$.
If $1, \frac{\theta_\mfp(\f)}{2\pi}, \frac{\theta_\mfp(\g)}{2\pi}$ are linearly independent over $\mathbb{Q}$, the sequence 
$\{(n\frac{\theta_\mfp(\f)}{2\pi},n\frac{\theta_\mfp(\g)}{2\pi})\}_{n\in \mathbb{N}}$ is uniformly distributed  $\pmod 1$ in $\R^2$
(cf.~\cite[Theorem 6.3]{KN74}). Now, the rest of the proof is similar to that of~\cite[Theorem 1.1]{Amr18}.
\end{proof}

In the above result, instead of $F$, if we work over $K$, then one can show the same result holds
for all but finitely many primes $\mfp \in \mathbb{P}$, instead of density $1$ set of primes $\mfp \in \mathbf{P}$.
In this case,  we can even drop  the assumption on $\f,\g$ being non-CM.

\bibliographystyle{plain, abbrv}

\begin{thebibliography}{abcd999}
\bibitem[Amr18]{Amr18}
       Amri, Mohammed Amin.
       Simultaneous sign change and equidistribution of signs of Fourier coefficients of two cusp forms. 
       Arch. Math. (Basel) 111 (2018), no. 3, 257--266.


\bibitem[BGG11]{BGG11}
        Barnet-Lamb, Thomas; Gee, Toby; Geraghty, David. 
        The Sato-Tate conjecture for Hilbert modular forms.
        J. Amer. Math. Soc. 24 (2011), no. 2, 411--469. 

\bibitem[Fre90]{Fre90}
        Freitag, Eberhard.
        Hilbert modular forms.
        Springer-Verlag, Berlin, 1990. 
        
\bibitem[Gar90]{Gar90}            
        Garrett, Paul B. 
        Holomorphic Hilbert modular forms.
        The Wadsworth \& Brooks/Cole Mathematics Series, CA, 1990.
        
\bibitem[GKP]{GKP}
        Gun, Sanoli; Kumar, Balesh; Paul, Biplab. 
        The First Simultaneous sign change and non-vanishing of Hecke Eigenvalues of newforms. 
        https://arxiv.org/pdf/1801.10590        
        
\bibitem[GKR15]{GKR15} 
        Gun, Sanoli; Kohnen, Winfried; Rath, Purusottam. 
 	Simultaneous sign change of Fourier-coefficients of two cusp forms. 
 	Arch. Math. (Basel) 105 (2015), no. 5, 413--424.
 	

\bibitem[KK]{KK} 
        Kaushik, Surjeet; Kumar, Narasimha.
        Simultaneous behaviour of the Fourier coefficients of two Hilbert modular cusp forms.
        Submitted.
        
\bibitem[KKT18]{KKT18}
        Kaushik, Surjeet; Kumar, Narasimha; Tanabe, Naomi. 
        Equidistribution of signs for Hilbert modular forms of half-integral weight.
        Res. Number Theory 4 (2018), no. 2, Art. 13, 10 pp.

\bibitem[KM14]{KM14}
        Kohnen, Winfried; Martin, Yves.
        Sign changes of Fourier coefficients of cusp forms supported on prime power indices.
         Int. J. Number Theory 10 (2014), no. 8, 1921--1927. 

\bibitem[KM18]{KM18}
        Kumari, Moni; Murty, M Ram.
        Simultaneous non-vanishing and sign changes of Fourier coefficients of modular forms.        
        To appear in  Int. J. Number Theory. 

         
\bibitem[KN74]{KN74}
        Kuipers, L.; Niederreiter, H. 
        Uniform distribution of sequences.
        Pure and Applied Mathematics. Wiley-Interscience [John Wiley \& Sons], New York-London-Sydney, 1974. 

  

\bibitem[KRW07]{KRW07} 
        Kowalski, Emmanuel; Robert, Olivier; Wu, Jie. 
        Small gaps in coefficients of $L$-functions and $B$-free numbers in short intervals. 
        Rev. Mat. Iberoam. 23 (2007), no. 1, 281--326.
        
        
 \bibitem[MKV18]{MKV18}
        Meher, Jaban; Shankhadhar, Karam Deo; Viswanadham, G. K. 
        On the coefficients of symmetric power $L$-functions. 
        Int. J. Number Theory 14 (2018), no. 3, 813--824.  

\bibitem[MM07]{MM07}        
        Ram Murty, M.; Kumar Murty, V. 
        Odd values of Fourier coefficients of certain modular forms. 
        Int. J. Number Theory 3 (2007), no. 3, 455--470.        

\bibitem[MT14]{MT14}
        Meher, Jaban; Tanabe, Naomi. Sign changes of Fourier coefficients of Hilbert modular forms.
        J. Number Theory 145 (2014), 230--244.
 
\bibitem[Mur83]{Mur83} 
       Murty, M. Ram. 
       Oscillations of Fourier coefficients of modular forms. 
       Math. Ann. 262 (1983), no. 4, 431--446.
 
\bibitem[Pan91]{Pan91} 
        Panchishkin, Alexey A. 
 	Non-Archimedean $L$-functions of Siegel and Hilbert modular forms. 
     	Lecture Notes in Mathematics, 1471. Springer-Verlag, Berlin, 1991.    
     	
\bibitem[Pri08]{Pri08}
        Pribitkin, Wladimir de Azevedo. 
        On the sign changes of coefficients of general Dirichlet series. 
        Proc. Amer. Math. Soc. 136 (2008), no. 9, 3089--3094.

\bibitem[Shi78]{Shi78} 
        Shimura, Goro. 
        The special values of the zeta functions associated with Hilbert modular forms. 
        Duke Math. J. 45 (1978), no. 3, 637-679.      
        
        
   
  
        

        



 
   
        

\end{thebibliography}

\end{document}